\documentclass[12pt]{amsart}

\usepackage{amsmath, enumerate, amsfonts, amssymb, amsthm, graphics, graphicx, verbatim, caption, subcaption, MnSymbol, booktabs, multirow, soul, array, hyperref, float, quiver, adjustbox, comment, todonotes}
\usepackage[numbers]{natbib}
\usepackage[table]{xcolor}
\usepackage[english]{babel}
\usepackage[margin=0.95in]{geometry}
\usepackage{tikz, float}
\usepackage{quiver}
\usepackage{mleftright}
\mleftright
\usetikzlibrary{shapes}
\usetikzlibrary{positioning}
\usetikzlibrary{decorations.markings}
\usetikzlibrary{arrows}
\usepackage{quiver}
\usepackage{tikz-cd}
\tikzstyle{subgroup}=[scale=1]

\newtheorem{theorem}{Theorem}[section]

\newtheorem{prevtheorem}{Theorem}

\newtheorem{corollary}[theorem]{Corollary}
\newtheorem{definition}[theorem]{Definition}

\newtheorem{lemma}[theorem]{Lemma}

\title{The category of centralizer lattices of groups}
\author[Cocke]{William Cocke}
\address{School of Computer and Cyber Sciences, Augusta University, Augusta, GA 30912; \newline \indent
wcocke@augusta.edu \\ \newline
\indent And Carnegie Mellon University, Pittsburgh, PA 15213: wcocke@andrew.cmu.edu}
\author[Lewis]{Mark L. Lewis}
\address{Department of Mathematical Sciences, Kent State University, Kent, OH  44242; \newline \indent lewis@math.kent.edu}
\author[McCulloch]{Ryan McCulloch}
\address{Department of Mathematics and Statistics, Binghamton University, Binghamton, NY 13902; \newline \indent rmccullo1985@gmail.com}
\date{}

\begin{document}

\begin{abstract}
We formalize the concept of a centralizer-respecting homomorphism, surjective homomorphisms which are equivariant with respect to taking the centralizer of a subgroup. There is a functor from the category of centralizer-respecting homomorphisms to the category of centralizer lattices. Finally, we conclude with some theorems about centralizer-respecting homomorphisms that show that the category of centralizer-respecting homomorphisms has many interesting maps.   
\end{abstract}



\maketitle

\section{Introduction.}

Within group theory, one encounters the concept of a centralizer, i.e., if $G$ is a group and $X$ a subset of elements of $G$, then we can form the set \[\mathbf{C}_G(X)=\{g\in G:  \forall x \in X, xg = gx\}.\] In words, $\mathbf{C}_G(X)$, called the centralizer of $X$ in $G$, is the set of all elements in $G$ that commute with every element of $X$, and for any $X\subseteq G$, the set $\mathbf{C}_G(X)$ is a subgroup of $G$. Of more interest, the centralizer operator acts on the subgroup lattice of $G$. As an operator, the centralizer has enough entertaining properties to make it an amusing object for further exploration. For example, the operator is antitone (short for anti-monotone), is an involution when restricted to its image, and behaves well with respect to intersections. \begin{flalign*}
     \textbf{(antitone)}&\;\; X \subseteq Y \Rightarrow \mathbf{C}_G(Y) \leq \mathbf{C}_G(X).\\ 
    \textbf{(involution on its image)}& \;\;\mathbf{C}_G(\mathbf{C}_G(\mathbf{C}_G(\cdot))) = \mathbf{C}_G(\cdot).\\ 
    \textbf{(intersection of centralizers)} & \;\;
    \mathbf{C}_G(X) \cap \mathbf{C}_G(Y) = \mathbf{C}_G(\langle X, Y\rangle). 
\end{flalign*}

Note that the notation $\langle X,Y \rangle$ means the subgroup generated by all elements in $X$ and $Y$.

The poset of all centralizer of a group $G$ forms the so-called centralizer lattice $\mathcal{C}(G)$ (see Definition \ref{def: cent-lat}). 

It is well-known that the centralizer behaves poorly with respect to homomorphisms, that is, for groups $G$ and $H$ with a homomorphism $\phi:G\to H$, it is often the case that $\phi(\mathbf{C}_G(A)) \neq \mathbf{C}_H(\phi(A))$ where $A$ is a subgroup of $G$. For example, if $G$ is nonabelian and $H$ is abelian then $\mathbf{C}_H(\phi(A))=H$ for all subgroups $A\leq G$. Since the centralizer does not respect homomorphisms, it is difficult to study in a categorical setting. This paper reconciles the centralizer operator and homomorphisms by formalizing the concept of centralizer-respecting homomorphisms. Moreover, we show that there is a pleasing categorical view of these homomorphisms by exhibiting a functor from centralizer-respecting homomorphisms to centralizer lattices. 

\begin{definition}[centralizer-respecting homomorphisms]
    A surjective group homomorphism {$\phi:G\rightarrow H$} is \textbf{centralizer-respecting} if for  subgroups $A\leq G$ we have $\phi(\mathbf{C}_G(A)) = \mathbf{C}_H(\phi(A))$.
\end{definition}

We show that centralizer-respecting homomorphisms induce maps on the corresponding centralizer lattices.  For a centralizer-respecting homomorphism $\phi$, we write $\mathcal{C}_\phi$ for the map that takes a centralizer $\mathbf{C}_G(X)$ to $ \mathbf{C}_H(\phi(X)) = \phi(\mathbf{C}_G(X))$. 

\begin{prevtheorem}\label{prevtheorem: hom}
    If $\phi$ is a centralizer-respecting homomorphism from $G$ to $H$, then the map $\mathcal{C}_\phi$ is a homomorphism of the centralizer lattices $\mathcal{C}(G) \rightarrow \mathcal{C}(H)$. 
\end{prevtheorem}

Moreover, more importantly, and more excitingly, the induced map is functorial. Once we properly define the category $\mathbf{Groups}_{\mathbf{{crh}}}$ of groups with centralizer-respecting homomorphisms and the category $\mathbf{CentLattices}$ of centralizer lattices in Section \ref{sec: categories}, we show that the map $\mathcal{C}$ given by $\mathcal{C}(G) = \mathcal{C}(G)$ and $\mathcal{C}(\phi)=\mathcal{C}_\phi$ is a functor.

\begin{prevtheorem}\label{prevtheorem: functor} The map $\mathcal{C}$ where $\mathcal{C}(G) = \mathcal{C}(G)$ and $\mathcal{C}(\phi) =\mathcal{C}_{\phi}$ is a functor from $\mathbf{Groups}_{\mathbf{crh}}$ to $\mathbf{CentLattices}.$    
\end{prevtheorem}

After proving Theorem \ref{prevtheorem: hom} and Theorem \ref{prevtheorem: functor} we address the important question of ``are there interesting centralizer-respecting homomorphisms?'', equivalently ``are there enough homomorphisms in the category $\mathbf{Groups}_{\mathbf{crh}}$ to explore?'' In Section \ref{sec: applications} we present a family of centralizer-respecting homomorphisms and show how to construct new centralizer-respecting homomorphisms from existing ones. 

\begin{prevtheorem}\label{prevtheorem: iff}
    Let $\phi:G\rightarrow K$ be a surjective homomorphism with a central kernel. The following are equivalent.
    \begin{enumerate}
        \item The homomorphism $\phi$ is centralizer-respecting.
        \item The kernel of $\phi$ does not contain any nontrivial commutators, elements of the form $[x,y] \neq 1$.
    \end{enumerate} 
    Moreover, if such a homomorphism $\phi$ is centralizer-respecting, then the induced map $\mathcal{C}_{\phi}$ is an isomorphism of centralizer lattices. 
\end{prevtheorem}

Theorem \ref{prevtheorem: iff} provides a number of examples of centralizer-respecting homomorphisms. For example, consider a group with exactly one non-commutator as identified by the recent work of Hatem and Sinoira \cite{hatem2026group}, or the family of such groups constructed by Skresanov \cite{skresanov2025finite}: the homomorphism corresponding to taking the quotient by the subgroup generated by the non-commutator is a centralizer-respecting homomorphism. Theorem \ref{prevtheorem: iff} also provides a most entertaining corollary as an example of application of the functor $\mathcal{C}$. 

\begin{corollary}\label{cor: iff}
    Fix a positive integer $n>2$. Let $G$ and $H$ be maximal class 2-groups of order $2^n$. Then the centralizer lattices of $G$ and $H$ are isomorphic.  
\end{corollary}

\begin{proof}
    As we show in Lemmata \ref{lem: D_Q} and \ref{lem: Q_SD}, any two nonisomorphic maximal class 2-groups $G$ and $H$ of order $2^n$ occur as quotients of a group $J$ of order $2^{n+1}$ by central subgroups $K$ and $I$ such that $K$ and $I$ do not contain any nontrivial elements of the form $[x,y]$. Hence $\mathcal{C}(J)\cong \mathcal{C}(G)$ and $\mathcal{C}(J)\cong \mathcal{C}(H)$. We conclude that $\mathcal{C}(G)\cong \mathcal{C}(H)$.
\end{proof}

The remainder of the paper proceeds as follows. In Section \ref{sec: groups/lattices} we provide some of the basics of group theory and lattice theory. This lets us prove Theorem \ref{prevtheorem: hom}. Next, we 
prove Theorem \ref{prevtheorem: functor} in Section \ref{sec: categories}. Finally, in Section \ref{sec: applications} we explore some of the homomorphisms that are in $\mathbf{Groups}_{\mathbf{crh}}$. 

\section{Groups and lattices.}\label{sec: groups/lattices}

Given a group, we can identify all of its subgroups into a poset. This poset is actually a lattice, i.e., a partially ordered set with a top $\top$ and bottom $\bot$ satisfying the rules of the following definition.

\begin{definition}[lattice]
    A \textbf{lattice} $L$ is a partially ordered set together with two distinguished elements $\top$ and $\bot$ such that for all $x \in L$ we have $x\leq \top$ and for all $x\in L$ we have $\bot \leq x$. There are two binary operations on $L$ known as meet, written $\wedge$, and join, written $\vee$; where 
    \begin{itemize} 
    \item $x \wedge y$ is the largest element $z$ of $L$ such that $z \leq x$ and $z\leq y$.
    \item $x \vee y$ is the smallest element $z$ of $L$ such that $x\leq z$ and $y \leq z$. 
    \end{itemize}
\end{definition}

\begin{definition}[lattice homomorphism]
    Given lattices $L, M$, a map $f:L\rightarrow M$ is a \textbf{lattice homomorphism} if $f(x) \wedge f(y) = f(x\wedge y)$ and $f(x) \vee f(y) = f(x\vee y)$. 
\end{definition}

For a group $G$, the set of all subgroups of $G$ form a lattice where $\leq$ is $\subseteq$ (inclusion as a subset), $\wedge = \cap$ (intersection) and $\vee = \langle \cdot \rangle$ (subgroup generated by), and $\top=G$ and $\bot = 1$. There is a rich history of investigating groups via their subgroup lattices, for example, the book by Schmidt \cite{schmidt1994subgroup} or the book by Suzuki \cite{suzuki2012structure}. More modern work on subgroup lattices tends to focus on specific subgroup lattices, e.g., the so-called Chermak--Delgado lattice \cite{wilcox2016exploring,lu2026finite} or other more exotic subgroup lattices \cite{BALLESTER-BOLINCHES_KAMORNIKOV_SHEMETKOVA_2026}.

For a group $G$, the poset of all of its centralizers form a lattice where the meet is intersection (as above) and the join is defined using the involution $\mathbf{C}_G(\cdot)$ as follows. 

\begin{definition}\label{def: cent-lat}(centralizer lattice)
    For a group $G$, let the \textbf{centralizer lattice} of $G$, written as $\mathcal{C}(G)$, consist of the poset of subgroups of $G$ of the form $\mathbf{C}_G(X)$ for some $X\subseteq G$, together with $\wedge_\mathbf{C}$ and $\vee_\mathbf{C}$, and an involution $\mathbf{C}_G(\cdot)$, where for all $X,Y \leq G$ we have \begin{flalign*}
        &\mathbf{C}_G(X)\wedge_\mathbf{C}  \mathbf{C}_G(Y) =\mathbf{C}_G(\langle X,Y\rangle) \\ 
        &\mathbf{C}_G(X) \vee_\mathbf{C}  \mathbf{C}_G(Y)  =\mathbf{C}_G\left(\mathbf{C}_G(\mathbf{C}_G(X)) \wedge_\mathbf{C} \mathbf{C}_G(\mathbf{C}_G(Y))\right)\\  &\mathbf{C}_G(\mathbf{C}_G(\mathbf{C}_G(X) =\mathbf{C}_G(X). \end{flalign*} 
\end{definition}

We use the notation $\wedge_{\mathbf{C}}, \vee_{\mathbf{C}}$ to distinguish from the operations $\wedge$ and $\vee$ on the subgroup lattice of $G$. Antonov \cite{antonov1994groups} has investigated groups where the operations coincide. For centralizers, $\mathbf{C}_G(X)\vee_\mathbf{C} \mathbf{C}_G(Y)$ is also the smallest centralizer in $G$ containing both $\mathbf{C}_G(X)$ and $\mathbf{C}_G(Y)$. While $\wedge_\mathbf{C}$ is just the restriction of $\wedge$ to inputs in $\mathcal{C}(G)$, it is not generally the case that \[\mathbf{C}_G(X) \vee_\mathbf{C} \mathbf{C}_G(Y) = \langle \mathbf{C}_G(X), \mathbf{C}_G(Y)\rangle = \mathbf{C}_G(X) \vee \mathbf{C}_G(Y).\]  

The centralizer lattice of a group is not a sublattice of the subgroup lattice because, as noted in Definition \ref{def: cent-lat}, the join operation is not the same as the join operation for the group itself. In addition, the centralizer lattice of a group $G$ has additional structure, i.e., the involution given by taking $\mathbf{C}_G(\cdot)$. 

\begin{definition}[centralizer lattice homomorphism]
    Given two centralizer lattices $\mathcal{C}(G)$ and $\mathcal{C}(H)$ a \textbf{centralizer lattice homomorphism} is a lattice homomorphism $f$ from $\mathcal{C}(G)$ to $\mathcal{C}(H)$ (over $\wedge_\mathbf{C}$ and $\vee_\mathbf{C}$)  that respects the centralizer involution, i.e., $f(\mathbf{C}_G(\mathbf{C}_G(X)) = \mathbf{C}_H(f(\mathbf{C}_G(X))$ for all $X \leq G$. 
\end{definition}

We now prove Theorem \ref{prevtheorem: hom} that for a centralizer-respecting homomorphism $\phi:G\to H$ the induced map $\mathcal{C}_\phi$ is a centralizer lattice homomorphism of $\mathcal{C}(G)\rightarrow \mathcal{C}(H)$. 

\begin{proof}[Proof of \,$\mathbf{Theorem}$ $\mathbf{A}$]
    The centralizer lattice is generated by the centralizer operator and either meet or joins. Hence to show that $\mathcal{C}_\phi$ is a homomorphism of centralizers lattices, we must show that $\mathbf{C}_{\phi}(\mathbf{C}_G(\cdot)) = \mathbf{C}_H(\mathcal{C}_{\phi}(\cdot))$ and that $\mathcal{C}_\phi$ respects either meets of joins. 
    
    For any $X\leq G$ we have  \[\mathcal{C}_{\phi}(\mathbf{C}_G(\mathbf{C}_G(X)) = \phi(\mathbf{C}_G(\mathbf{C}_G(X))) = \mathbf{C}_H(\phi(\mathbf{C}_G(X))) = \mathbf{C}_H(\mathcal{C}_{\phi}(\mathbf{C}_G(X))),\] and conclude that $\mathcal{C}_\phi$ preserves the centralizer operator. Chasing the diagram in Figure \ref{fig:hom_map} shows that $\mathcal{C}_\phi$ preserves meets (intersections) and is thus a lattice homomorphism. To chase the diagram consider the explicit calculation.  

\begin{figure}
    \centering
        \begin{tikzpicture}[
  node distance=2.5cm and 3.5cm,
  every node/.style={font=\small},
  arrow/.style={->, thick},
  dashedarrow/.style={->, dashed}
]

\node (X) {$\mathbf{C}_G(X)$};
\node (Y) [right=of X] {$\mathbf{C}_G(Y)$};

\node (XY) [below=of $(X)!0.5!(Y)$] 
  {$\mathbf{C}_G(X)\cap \mathbf{C}_G(Y)=\mathbf{C}_G(\langle X,Y\rangle)$};

\node (phiX) [right=.5cm of XY, yshift=1.5cm] 
  {$\phi(\mathbf{C}_G(X))=\mathbf{C}_H(\phi(X))$};

\node (phiY) [right=6cm of XY, yshift=1.5cm] 
  {$\phi(\mathbf{C}_G(Y))=\mathbf{C}_H(\phi(Y))$};

\node (phiXY) [below=2.5cm of XY] 
  {$\phi(\mathbf{C}_G(\langle X,Y\rangle))$};

\node (CphiXY) [right=-.1cm of phiXY] 
  {$    = \mathbf{C}_H(\phi(\langle X,Y\rangle)) 
    = \mathbf{C}_H(\langle \phi(X),\phi(Y)\rangle)=$};
    
\node (phiInt) [right=-.1cm of CphiXY]
  {$\mathbf{C}_H(\phi(X)) \cap \mathbf{C}_H(\phi(Y))$};

\draw[arrow] (X) -- (XY);
\draw[arrow] (Y) -- (XY);

\draw[arrow] (phiX) -- (phiInt);
\draw[arrow] (phiY) -- (phiInt);

\draw[dashedarrow, bend left=6] (X) to node[pos=.35, above] {$\mathcal{C}_\phi$} (phiX);
\draw[dashedarrow, bend left=6] (Y) to node[above] {$\mathcal{C}_\phi$} (phiY);

\draw[dashedarrow] (XY) -- node[right] {$\mathcal{C}_\phi$} (phiXY);

\end{tikzpicture}
    \caption{Diagram showing that $\mathcal{C}_{\phi}$ preserves intersections.}
    \label{fig:hom_map}
\end{figure}
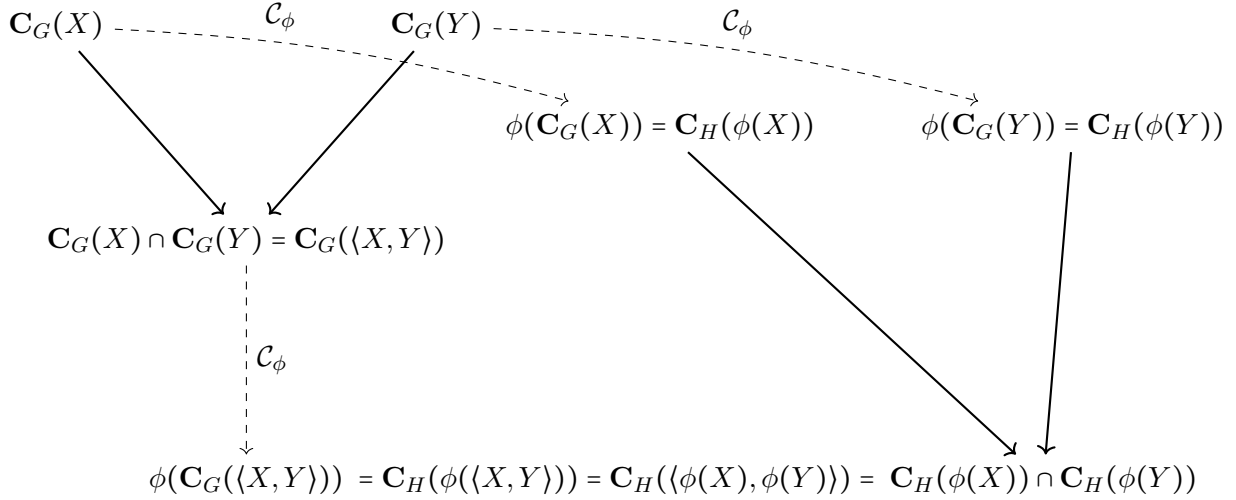

    \begin{flalign*}
        \mathcal{C}_\phi(\mathbf{C}_G(X) \cap \mathbf{C}_G(Y)) &=\mathcal{C}_\phi(\mathbf{C}_G(\langle X,Y \rangle)) \\&= \phi(\mathbf{C}_G(X,Y))\\
        &= \mathbf{C}_H(\phi(\langle X, Y))\\
        &= \mathbf{C}_H(\langle \phi(X), \phi(Y) \rangle)\\
        &= \mathbf{C}_H(\phi(X)) \cap \mathbf{C}_H(\phi(Y))\\
        &=\phi(\mathbf{C}_G(X)) \cap \phi(\mathbf{C}_G(Y)) \\
        &= \mathcal{C}_{\phi}(\mathbf{C}_G(X)) \cap \mathcal{C}_{\phi}(\mathbf{C}_G(Y))
    \end{flalign*}

    We conclude that $\mathcal{C}_\phi$ is a centralizer lattice homomorphism from $\mathcal{C}(G)$ to $\mathcal{C}(H)$.

\end{proof}

\section{The two categories and a functor between them.} \label{sec: categories}
We are now ready to show that groups with centralizer-respecting homomorphisms forms a category. We use the definitions of categories and functors from Awodey \cite{awodey2010category}.



    

\begin{lemma} \label{lem: crh_category}
    There is a category whose objects consist of groups and whose arrows are centralizer-respecting homomorphisms. We write $\mathbf{Groups}_{\mathbf{crh}}$ for this category. 
\end{lemma}
\begin{proof}
    Since $\mathbf{Groups}$ forms a category, we know that the composition of homomorphisms is associative, and there are appropriate identity homomorphisms from each group to itself. To show that $\mathbf{Groups}_{\mathbf{crh}}$ forms a category, we need to show that the composition of two centralizer-respecting homomorphisms is a centralizer-respecting homomorphism. Let both $\phi:G\rightarrow H$ and $\psi:H\rightarrow K$ be centralizer-respecting homomorphisms. Then  
    \[(\psi\circ \phi)(\mathbf{C}_G(A)) = \psi(\mathbf{C}_H(\phi(A))) = \mathbf{C}_K(\psi(\phi(A))).\] Since the composition of surjective maps is surjective, we conclude that the composition of centralizer-respecting homomorphisms is centralizer-respecting. 

    Moreover, for any group $G$, the identity isomorphism is the identity arrow $\mathbf{1}_G$.

    Hence $\mathbf{Groups}_{\mathbf{crh}}$ forms a category.
\end{proof}

\begin{definition}[$\mathbf{CentLattices}$]
The category $\mathbf{CentLattices}$ has centralizer lattices of groups as objects and centralizer lattice homomorphisms as arrows.  
\end{definition}

\begin{definition}[functor]
    A \textbf{functor} $\mathcal{F}$ from $\mathcal{A}$ to $\mathcal{B}$ consists of \begin{itemize}
        \item a mapping of objects $A\in \mathcal{A}$ to objects $\mathcal{F}(A) \in \mathcal{B}$. 
        \item a mapping of morphisms $f \in \mathcal{A}$ to morphisms $\mathcal{F}(f)\in \mathcal{B}$ such that:
        \begin{enumerate}
            \item $\text(dom)(\mathcal{F}(f)) = \mathcal{F}(\text{dom}(f))$.
            \item $\text(cod)(\mathcal{F}(f)) = \mathcal{F}(\text{cod}(f))$.
            \item $\mathcal{F}(1_A)  = 1_{\mathcal{F}(A)}$.
            \item $\mathcal{F}(g\circ f) = \mathcal{F}(g) \circ \mathcal{F}(f)$, for composable $f$ and $g$ in $\mathcal{A}$.
        \end{enumerate}
    \end{itemize}
\end{definition}

We now prove Theorem \ref{prevtheorem: functor} that the map $\mathcal{C}$ is a functor from $\mathbf{Groups}_{\mathbf{crh}}$ to $\mathbf{CentLattices}$. 

\begin{proof}[Proof of\, $\mathbf{Theorem}$ $\mathbf{B}$]
We need to show that $\mathcal{C}$ satisfies the definition of a functor.

By construction $\mathcal{C}$ maps groups to centralizer lattices. Moreover for a centralizer-respecting homomorphism $\phi$ from $G$ to $H$ the centralizer lattice homomorphism $\mathbf{C}_\phi$ has domain $\mathcal{G}$ and $\mathcal{H}$. 

We note that for the identity homomorphism $1_G$ from $G$ to $G$, the induced map on centralizer lattices is the identity map, i.e., $\mathcal{C}(1_G)=1_{\mathcal{C}(G)}$. 

We now show that $\mathcal{C}$ behaves well with respect to composition of homomorphisms. Let $\phi:G\rightarrow H$ and $\psi:H\rightarrow K$ be centralizer-respecting homomorphisms. Consider the diagram in Figure \ref{fig: composition}. We compute  \[(\mathcal{C}_{\psi}\circ \mathcal{C}_{\phi})(\mathbf{C}_G(X)) = \mathcal{C}_\psi(\phi(\mathbf{C}_G(X))) = \psi(\phi(\mathbf{C}_G(X))) = (\psi \circ \phi)(\mathbf{C}_G(X)) = \mathcal{C}_{\psi\circ \phi}(\mathbf{C}_G(X)). \]
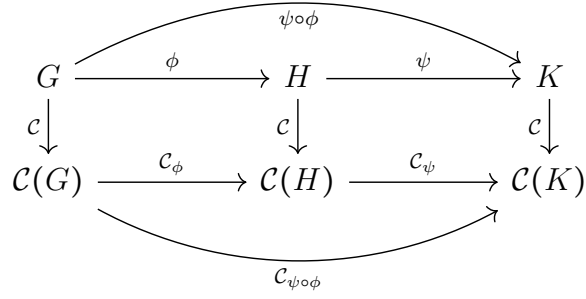
\begin{figure}
    \centering
    
\[\begin{tikzcd}
	G && H && K \\
	{\mathcal{C}(G)} && {\mathcal{C}(H)} && {\mathcal{C}(K)}
	\arrow["\phi", from=1-1, to=1-3]
	\arrow["{\psi \circ \phi}"', curve={height=-36pt}, from=1-1, to=1-5]
	\arrow["{\mathcal{C}}"', from=1-1, to=2-1]
	\arrow["\psi", from=1-3, to=1-5]
	\arrow["{\mathcal{C}}"', from=1-3, to=2-3]
	\arrow["{\mathcal{C}}"', from=1-5, to=2-5]
	\arrow["{\mathcal{C}_{\phi}}", from=2-1, to=2-3]
	\arrow["{\mathcal{C}_{\psi \circ \phi}}"', curve={height=36pt}, from=2-1, to=2-5]
	\arrow["{\mathcal{C}_{\psi}}", from=2-3, to=2-5]
\end{tikzcd}\]
    \caption{The commutative diagram showing that $\mathcal{C}$ is a functor from $\mathbf{Groups}_{\mathbf{crh}}$ to $\mathbf{CentLattices}.$}
    \label{fig: composition}
\end{figure}

Hence $\mathcal{C}$ is a functor from $\mathbf{Groups}_{\mathbf{crh}}$ to $\mathbf{CentLattices}$. 

\end{proof}

\section{Examples of centralizer-respecting homomorphisms.}\label{sec: applications}

In this section, we show the existence of many centralizer-respecting homomorphisms. As part of this discussion we note the following cancellation property which is of independent interest. 


 
    


\begin{lemma}\label{lem: cancellation}
    Let $G,H,K$ be groups and let $\phi:G\rightarrow H$ be centralizer-respecting and $\psi:H\rightarrow K$ be surjective. Then $\psi\circ \phi $ is centralizer-respecting if and only if $\psi$ is centralizer-respecting.
\end{lemma}
\begin{proof}
    $\Rightarrow$ Suppose by way of contradiction that $\psi \circ \phi$ is centralizer-respecting, but $\psi$ is not centralizer-respecting. Then there is some subgroup $A\leq H$ such that $\psi(\mathbf{C}_H(A)) \neq \mathbf{C}_K(\psi(A))$. Since $\phi$ is surjective, we can select a $B\leq G$ such that $\phi(B)=A$. Then we obtain 
    \begin{flalign*} \mathbf{C}_K(\psi \circ \phi (B)) &= (\psi\circ \phi)(\mathbf{C}_G(B)) = \psi ( \mathbf{C}_H(\phi(B))) = \psi(\mathbf{C}_H(A)).\\
    \mathbf{C}_K(\psi \circ \phi(B)) &= \mathbf{C}_K(\psi (\phi(B))) =\mathbf{C}_K(\psi(A)).  \end{flalign*} Contradicting the assumptions that $\psi$ is not centralizer-respecting. 
    
    $\Leftarrow$ The composition of centralizer-respecting homomorphisms is centralizer-respecting since $\mathbf{Groups}_{\mathbf{cr
    h}}$ forms a category. 
\end{proof}

    

The next lemma demonstrates that the functor $\mathcal{C}$ translates some properties of centralizer-respecting homomorphisms to the induced map on lattices. The ability to ask such questions motivates the categorical approach. 

\begin{lemma}
    Let $\phi$ be a centralizer-respecting homomorphism between two groups $G$ and $H$. Then the kernel of $\phi$ is central if and only if the induced map $\mathcal{C}_{\phi}$ on centralizer lattices is an isomorphism. 
\end{lemma}
\begin{proof}
   $\Rightarrow$ Write $K$ for the kernel of $\phi$ and assume that $K$ is central. We show that $\mathcal{C}_\phi$ is an isomorphishm. First we show it is surjective. Given a centralizer $\mathbf{C}_H(B)$ in $H$, we note that there is some $A\leq G$ with $\phi(A) = B$. Then \[\mathbf{C}_H(B) = \mathbf{C}_H(\phi(A)) = \phi(\mathbf{C}_G(A))= \mathcal{C}_\phi(\mathbf{C}_G(A)).\] For injectivity consider two subgroups $A,C \leq G$ such that $\mathcal{C}_{\phi}(\mathbf{C}_G(A)) = \mathcal{C}_\phi (\mathbf{C}_G(C))$. Then $\phi(\mathbf{C}_G(A) = \phi(\mathbf{C}_G(C))$ which implies that $\mathbf{C}_G(A) K = \mathbf{C}_G(C) K$. But, since $K$ is central, it is contained in $\mathbf{C}_G(A)$ and $\mathbf{C}_G(C)$ and thus we conclude that $\mathbf{C}_G(A) = \mathbf{C}_G(C)$. 

   $\Leftarrow$ Assume that $\mathcal{C}_\phi$ is an isomorphism of $\mathcal{C}(G)$ and $\mathcal{C}(H)$. Then in particular \[H=\mathbf{C}_H(\phi(K))= \phi(\mathbf{C}_G(K)).\] Thus $\mathbf{C}_G(K) = G$ and $K$ is central.
\end{proof}


We can now prove Theorem \ref{prevtheorem: iff} which establishes a sufficient and necessary criteria for surjective homomorphisms with central kernels to be centralizer-respecting. 

\begin{proof}[Proof of $\mathbf{Theorem}$ $\mathbf{\ref{prevtheorem: iff}}$]
(1) $\Rightarrow (2)$. Suppose that the homomorphism $\phi$ is centralizer-respecting and by way of contradiction that $\phi([a,b])=1$ for $a,b\in G$ and $[a,b]\neq 1$. Then $b$ is not contained in $\mathbf{C}_G(\langle a \rangle )$. So $[\phi(a),\phi(b)] = \phi([a,b]) = 1$, which implies that $\phi(b)$ is contained in $\mathbf{C}_H(\phi(\langle a\rangle)) = \phi(\mathbf{C}_G(\langle a \rangle)$. Hence we can write $b$ as $c z$ for some $c\in \mathbf{C}_G(\langle a \rangle)$ and $z\in \mathbf{ker}(\phi)$, where $\mathbf{ker}(\phi)$ is central. Then $[a,b] = [a,cz] = 1$ a contradiction. 

(2) $\Rightarrow (1)$. Suppose now that the $\mathbf{ker}(\phi)$ does not contain any nontrivial commutators and by way of contradiction that the homomorphism $\phi$ is not centralizer-respecting. Choose an $A\leq G$ such that $\phi(\mathbf{C}_G(A)) \neq \mathbf{C}_H(\phi(A))$. There is an element $b\in G$ with $[a,b]\neq 1$ such that $\phi(b) \in \mathbf{C}_H(\phi(A))$. Hence $\phi([a,b])=1$ a contradiction.
\end{proof} 

Figure \ref{fig:example} represents the homomorphism from the group $G=\langle x, y \; | \; x^y = x^{-1} \rangle$ to $Q_8$ and shows the induced isomorphism of the centralizer lattice.

We conclude by proving the lemmata mentioned in Corollary \ref{cor: iff}, which stated that any two maximal class 2-groups of the same order have isomorphic centralizer lattices. Since the corollary occurs earlier in the paper, we emphasize that the following lemmata are independent of the rest of the paper.
\begin{lemma}\label{lem: D_Q}
Fix $n>2$. Let $G = \langle x, y\; | \;x^{2^{{n-1}}}=y^{4} =1, x^y = x^{-1} \rangle $ of order $2^{n+1}$. Then the following subgroups of $G$ are central and do not contain any nontrivial elements of the form $[g,h]$ for $g,h\in G$. 
\begin{itemize}
    \item $Z_D = \langle y^2 \rangle.$ 
    \item $Z_Q = \langle x^{2^{n-2}} y^2 \rangle.$
\end{itemize} In addition the quotient $G/Z_D$ is a dihedral group of order $2^n$ and the quotient $G/Z_Q$ is a generalized quaternion group of order $2^n$.   
\end{lemma}
\begin{proof} 
Since $x^{y^2} = (x^{-1})^{-1} = x$, we conclude that $y^2$ is central and has order 2. Similarly $(x^{2^{n-2}})^y = ({x^{2^{n-2}}})^{-1} = x^{2^{n-2}}$ and thus $x^{2^{n-2}}$ is central. Thus $Z_D$ and $Z_Q$ are both central subgroups of $G$ and each have order $2$. The commutator subgroup of $G$ is contained in $\langle x \rangle$ since $G/\langle x \rangle$ is abelian. Hence neither $Z_D$ nor $Z_Q$ contain any nontrivial commutators. 

We now examine the quotient $G/Z_D$, which has presentation, using the standard bar-notation for elements of the quotient: \[G/Z_D = \langle \overline{x}, \overline{y} \; | \; \overline{x}^{2^{n-1}} = \overline{y}^2 = 1, \overline{x}^{\overline{y}} = \overline{x}^{-1}\rangle,\] which is the standard presentation for the dihedral group of order $2^{n}$. 

Consider $G/Z_Q$, again with the standard bar-notation for elements of the quotient: 
\[G/Z_Q = \langle \overline{x},\overline{y}\;|\; \overline{x}^{2^{n-1}} = \overline{y}^{4} = 1, \overline{x}^{2^{n-2}}=\overline{y}^2, \,\overline{x}^{\overline{y}} = \overline{x}^{-1}\rangle ,\] which is the standard presentation for the quaternion group of order $2^n$. 
\end{proof}

\begin{lemma}\label{lem: Q_SD}
Fix $n>2$. Let $G =\langle x,y\,|\, x^{2^{2n-1}}=y^4 = 1, x^y = x^{2^{n-2}-1}\rangle$ of order $2^{n+1}$. 
The following subgroups of $G$ are central and do not contain any nontrivial elements of the form $[g,h]$ for $g,h\in G$. 
\begin{itemize} 
\item $Z_Q = \langle  x^{2^{n-2}} y^2 \rangle$.
\item $Z_{SD} = \langle y^2 \rangle$.
\end{itemize} In addition the quotient $Z_Q$ is a generalized quaternion group of order $2^n$ and the quotient $G/Z_{SD}$ is a semi-dihedral group of order $2^n$. 
\end{lemma}
\begin{proof}
    Since $x^{y^2} = (x^{2^{n-2}-1})^{2^{n-2}-1} = x$, we conclude that $y^2$ is central and has order 2. Similarly $({x^{2^{n-2}}})^y = x^{2^{n-2}}$ and thus $x^{2^{n-2}}$ is central. Thus $Z_Q$ and $Z_{SD}$ are both central subgroups of $G$ and each have order $2$. The commutator subgroup of $G$ is contained in $\langle x \rangle$ since $G/\langle x \rangle$ is abelian. Hence neither $Z_D$ nor $Z_Q$ contain any nontrivial commutators. 

Consider $G/Z_Q$, again with the standard bar-notation for elements of the quotient: 
\[G/Z_Q = \langle \overline{x},\overline{y}\;|\; \overline{x}^{2^{n-1}} = \overline{y}^{4} = 1, \overline{x}^{2^{n-2}}=\overline{y}^2, \overline{x}^{\overline{y}} = \overline{x}^{2^{n-2}-1}\rangle ,\] which is not the standard representation of the quaternion group. Write  $\overline{z}$ for $\overline{xy^2}$. Then $\overline{z}^{\overline{y}} = \overline{x}^{2^{n-2}}\overline{x}^{-1} = \overline{y}^2 \overline{x}^{-1} = \overline{z}^{-1}$. This allows us to write
\[G/Z_Q = \langle \overline{z}, \overline{y} \; | \; \overline{z}^{2^{n-1}} = \overline{y}^4 = \overline{1}, \,\overline{z}^{2^{n-2}} = \overline{y}^2,\, \overline{z}^{\overline{y}} = \overline{z}^{-1} \rangle,\] which is the standard presentation of the quaternion group of order $2^n$. 

We now explore $G/Z_{SD}$, again with the standard bar-notation for elements of the quotient:
\[
G/Z_{SD} = \langle \overline{x}, \overline{y} \; | \; \overline{x}^{2^{2n-1}} = \overline{y}^2 = 1,\, \overline{x}^{\overline{y}} = \overline{x}^{2^{n-2}-1},\rangle 
\] which is the standard presentation of the semi-dihedral group of order $2^n$. 
\end{proof}

    Figure \ref{fig:example} shows an isomorphism from the group $G= \langle x,y \; | \; x^4=y^4=1, x^y=x^{-1}\rangle$ to $Q_8$. 

\begin{figure}
    \centering
\[\begin{tikzcd}[cramped, sep=small]
	&& G &&&&& {G/K} & \\
	&& \node[draw=gray] (2-3) {\mathcal{C}(G)}; &&&&& \node[draw=gray] (2-8) {\mathcal{C}(G/K)};\\[3em]
	&& &&&&& \node[draw=gray] (3-8) {Q_8}; \\
	&&&&&& \node[draw=gray] (4-7) {I}; & \node[draw=gray] (4-8) {J}; & \node[draw=gray] (4-9) {K}; \\
	&& \node[draw=gray] (5-3) {G=\mathbf{Z}_4\rtimes \mathbf{Z}_4}; &&&&& \node[draw=gray] (5-8) {\langle-1\rangle}; \\
	& \node[draw=gray] (6-2) {\mathbf{Z}_4\times \mathbf{Z}_2}; && \node[draw=gray] (6-4) {\mathbf{Z}_2\times\mathbf{Z}_4}; && \node[draw=gray] (6-6) {\mathbf{Z}_2\times\mathbf{Z}_4}; && 1 \\
	{\mathbf{Z}_4} & {\mathbf{Z}_4} & \node[draw=gray] (7-3) {\mathbf{Z}_2\times\mathbf{Z}_2}; & {\mathbf{Z}_4} & {\mathbf{Z}_4} & {\mathbf{Z}_4} & {\mathbf{Z}_4} \\
	& {\mathbf{Z}_2} & {K=\mathbf{Z}_2} && {\mathbf{Z}_2} \\
	&& 1
	\arrow["\phi", from=1-3, to=1-8]
	\arrow["{\mathcal{C}_\phi}", dotted, from=2-3, to=2-8]
	\arrow[no head, from=3-8, to=4-8]
	\arrow[no head, from=3-8, to=4-9]
	\arrow[no head, from=4-7, to=3-8]
	\arrow[no head, from=4-9, to=5-8]
	\arrow[dotted, from=5-3, to=3-8]
	\arrow[no head, from=5-3, to=6-4]
	\arrow[no head, from=5-3, to=6-6]
	\arrow[no head, from=5-8, to=4-7]
	\arrow[no head, from=5-8, to=4-8]
	\arrow[dotted, from=6-2, to=4-7]
	\arrow[no head, from=6-2, to=5-3]
	\arrow[dotted, from=6-4, to=4-8]
	\arrow[no head, from=6-4, to=7-3]
	\arrow[ no head, from=6-4, to=7-4]
	\arrow[curve={height=-12pt}, dotted, from=6-6, to=4-9]
	\arrow[no head, from=6-6, to=7-3]
	\arrow[ no head, from=6-6, to=7-6]
	\arrow[ no head, from=6-8, to=5-8]
	\arrow[ no head, from=7-1, to=6-2]
	\arrow[no head, from=7-1, to=8-2]
	\arrow[ no head, from=7-2, to=6-2]
	\arrow[no head, from=7-2, to=8-2]
	\arrow[curve={height=115pt}, dotted,  from=7-3, to=5-8]
	\arrow[no head, from=7-3, to=6-2]
	\arrow[ no head, from=7-3, to=8-2]
	\arrow[ no head, from=7-3, to=8-3]
	\arrow[no head, from=7-4, to=8-5]
	\arrow[ no head, from=7-5, to=6-4]
	\arrow[no head, from=7-5, to=8-5]
	\arrow[no head, from=7-6, to=8-5]
	\arrow[ no head, from=7-7, to=6-6]
	\arrow[no head, from=8-2, to=9-3]
	\arrow[ no head, from=8-3, to=9-3]
	\arrow[ no head, from=8-5, to=7-3]
	\arrow[no head, from=8-5, to=7-7]
	\arrow[no head, from=8-5, to=9-3]
\end{tikzcd}\]
    \caption{\sloppy Diagram of the subgroup lattice of the group $G=\mathbf{Z}_4\rtimes  \mathbf{Z}_4$ where the action is by inversion. We can also write $G= \langle x,y \; | \; x^4=y^4=1, x^y=x^{-1}\rangle$, the subgroup lattice of $G/K$ where $K=\langle x^2y^2\rangle$, and the centralizer lattice of $G$ and $G/K$ are shown via the boxed subgroups of the respective lattices. The map $\phi:G\rightarrow G/K$ given by taking the quotient is centralizer-respecting and thus we have an induced map $\mathcal{C}_{\phi}$ on the centralizer lattice, shown via the dashed lines.}
    \label{fig:example}
\end{figure}
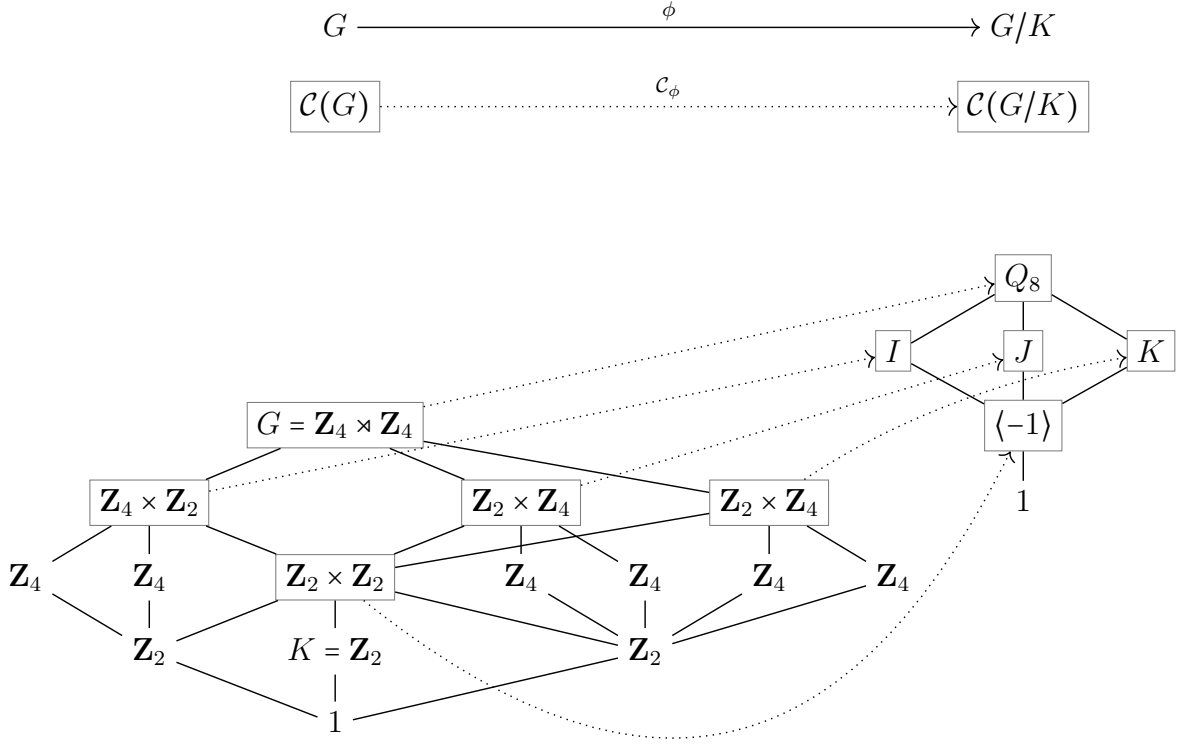

\bibliographystyle{acm}
\bibliography{ref}

\end{document}